\documentclass{article}
\usepackage{graphicx}
\usepackage{pdfpages}
\usepackage[english]{babel}
\usepackage{amsmath}
\usepackage{amsthm}
\usepackage{physics}
\usepackage[a4 paper]{geometry}
\usepackage{graphicx}
\usepackage{frontespizio}
\usepackage{mathtools}
\usepackage{amsfonts}
\usepackage{enumitem}
\usepackage{hyperref}
\usepackage{empheq}
\usepackage{bm}
\usepackage[T1]{fontenc}
\usepackage[utf8]{inputenc}
\usepackage[autostyle]{csquotes}
\usepackage[backend=biber, maxnames=10]{biblatex}
\usepackage{amssymb}
\theoremstyle{definition}
\newtheorem*{prob*}{Problem}
\newtheorem{prob}{Problem}
\newtheorem{cor}{Corollary}
\newtheorem{defi}{Definition}
\newtheorem{prop}{Proposition}
\newtheorem{rmk}{Remark}
\newtheorem*{rmk*}{Remark}

\theoremstyle{plain}
\newtheorem{teo}{Theorem}
\newtheorem{lem}{Lemma}
\newcommand{\h}{\mathbf{H}}
\newcommand{\q}{\mathbf{Q}}
\newcommand{\m}{\mathbf{M}}

\newcommand{\rot}{\operatorname{curl}}
\newcommand{\di}{\operatorname{div}}

\newcommand{\e}{\mathbf{E}}
\newcommand{\n}{\mathbf{n}}
\newcommand{\J}{\mathbf{J}}
\newcommand{\la}{\langle}
\newcommand{\ra}{\rangle}

\newcommand{\ov}{\overline{\mathbf{v}}}
\newcommand{\R}{\mathbb{R}}
\newcommand{\C}{\mathbb{C}}
\newcommand{\p}{\mathbf{p}}
\newcommand{\widep}{\widehat{\mathbf{p}}}
\numberwithin{equation}{section}
\title{Optimal Control of an Eddy Current Problem \\ with a Dipole Source}
\date{}

\author{\normalsize{\textsc{Gabriele Caselli}} \thanks{Department of Mathematics, University of Trento, Via Sommarive 14, Povo (TN), 38123 Italy ({gabriele.caselli@unitn.it}).}}
\begin{document}
\maketitle
\begin{abstract}
    
This paper is concerned with the analysis of a class of optimal control problems governed by a time-harmonic eddy current system with a dipole source, which is taken as the control variable. A mathematical model is set up for the state equation where the dipole source takes the form of a Dirac mass located in the interior of the conducting domain. A non-standard approach featuring the fundamental solution of a $\rot \rot - \mathrm{Id}
$ operator is proposed to address the well-posedness of the state problem, leading to a split structure of the state field as the sum of a singular part and a regular part. The aim of the control is the best approximation of desired electric and magnetic fields via a suitable $L^2$-quadratic tracking cost functional. Here, special attention is devoted to establishing an adjoint calculus which is consistent with the form of the state variable and in this way first order optimality conditions are eventually derived.

\end{abstract}
\section{Introduction}

The aim of devising optimal control procedures for Maxwell's equations and eddy current systems is not new in itself, considering the important role of electromagnetic fields in various modern technologies. Once suitable mathematical tools became available in the literature\footnote{The very first exhaustive characterization of the traces of $\bm{H}(\rot)$ functions in rather general domains came out in the early 2000\textit{s} with \cite{buffa2002traces} by Buffa \textit{et al.}} many researchers have started focusing their attention on this kind of problems, most of the times considering distributed controls in the form of a current density in the interior of a conducting domain, or in the form of a voltage excitation on the boundary (i.e., via electric ports): we refer to Tr\"{o}ltzsch and Valli \cite{troltzsch2016optimalSparse}, \cite{troltzschValliQuarto} and to Yousept \cite{youseptcontroleddy} for linear time-harmonic eddy current problems, and to Tr\"{o}ltzsch and Valli \cite{troltzschvalliMCRF} or Nicaise \textit{et al.} \cite{nicaise2014two} for the time-dependent case. We also mention the work of Bommer and Yousept \cite{bommer2016optimal} featuring the full Maxwell system as well as the one of Yousept \cite{yousept2013optimal} where a quasi-linear case is investigated.
\par 
At the same time, there are several applied contexts in which one is interested in finding an
optimal way to \textit{place} sensors or actuators; along with this, if it is not possible - or not enough efficient -  to distribute control devices all over the domain, the problem of identifying which sub-regions are actually important in order to achieve the minimization of the objective functional arises. Following the archetype work of Stadler \cite{stadler2009elliptic}, it became clear that the addition of a non-smooth $L^1$-regularization term in the cost functional entails sparse properties of the optimal solutions, namely that they have small support with respect to the Lebesgue measure. These techniques have already been applied, though not extensively, in the context electromagnetic PDEs, see for instance Tr\"{o}ltzsch and Valli \cite{troltzsch2016optimalSparse} and the author \cite{TesiGabriele}. \par
In more recent times, the lack of reflexivity, compactness and differentiability (regularity) properties of the $L^1$-spaces and norms led to the study of optimal control problems in measure spaces like $\mathcal{M}(\Omega)$, the space of regular Borel measures, or $L^2(I, \mathcal{M}(\Omega))$, which both exhibit better functional properties as well as similar sparsity features of optimal solutions; see Casas \textit{et al.} \cite{casas2013parabolic}, Clason and Kunish \cite{clason2011duality} and Trautmann \textit{et al.} \cite{trautmann2017finite}, where all these issues are widely discussed. \par
Let us focus our attention on controls of the form:
        \begin{equation} \label{linear combination of Deltas}
            u = \sum_{i=1}^{N} u_i \delta_{x_i}, \quad x_i \in \Omega,
        \end{equation}
where $u_i$ is, say, either a complex number or a time-dependent intensity $t \mapsto u_i(t)$.  
These are typical examples of singular elements in $\mathcal{M}(\Omega)$ (respectively in $L^2(I; \mathcal{M}(\Omega))$ in the time-dependent framework) that are usually of interest for modelling phenomena related to geology or acoustics: we refer to Pieper \textit{et al.} \cite{pieper2018inverse}, where an inverse problem from point-wise measurements (state observations) is analyzed; nevertheless, a work by Alonso Rodr\'iguez \textit{et al.} \cite{rodriguez2011inverse} concerning inverse problems for eddy current equations suggests that sources (controls) of type $\eqref{linear combination of Deltas}$ can be meaningful also for electromagnetic problems: a weighted Dirac mass $\mathbf{p} \delta_{\mathbf{x}_0}$ represents a dipole source of intensity $\mathbf{p} \in \mathbb{R}^3$ concentrated at $\mathbf{x}=\mathbf{x}_0$. \par
In principle, this would lead to consider controls that can be a priori expressed as a
linear combination of deltas with \textit{unknown positions} and \textit{unknown intensities}, an assumption which, in turn, yields a
non-convex optimization problem. A common idea to overcome this difficulty is precisely to lift the problem to a more general one
with controls lying in a suitable space of measures, and \textit{then} discuss if and under
what conditions the solution has the desired structure $\eqref{linear combination of Deltas}$. \par However, the latter step is far from being reliable, often providing just some necessary conditions and/or information on the support of the optimal measure. For what concerns electromagnetic state equations, the situation is even more complicated since the analysis of PDEs with measure-valued sources usually requires some structural regularity of the differential operator, while Maxwell's equations naturally exhibit singular solutions in most instances; see e.g. Costabel \textit{et al.} \cite{costabel2003singularities}. For these reasons, we decided to work with a fixed number of deltas (i.e., one, without loss of generality) in a fixed location, say $\mathbf{x}=\mathbf{x}_0$. A similar approach has been carried out rather recently by Allendes \textit{et al.} \cite{allendes2018posteriori}, but there the state equation takes the form of a Poisson problem and the focus is shifted on the a posteriori error analysis for a FEM approximation.    \par
Despite the adopted simplifications, several mathematical difficulties are here present: the most important, as mentioned, is that our state equation is an eddy current system with a Dirac distribution as source. We propose an approach that seems new in this context; the resolution of the problem is split into three steps, the first one being the determination of a fundamental solution to deal with the singularity at $\mathbf{x}_0$ (this idea has been already used to tackling some inverse problems; see for instance Wolters \textit{et al.} \cite{MR2377429} and Alonso Rodr\'iguez \textit{et al.} \cite{rodriguez2011inverse}). After that, the specific structure of the eddy current problem leads to a state variable that is composed by two terms, a vector one and (the gradient of) a scalar one. The control analysis inherits these issues and thus two adjoint states, corresponding to two different \textit{parts} of the state variable, need to be defined in a non-standard way. Moreover, the underlying complex structure of the spaces involved in the analysis of time-harmonic Maxwell's equations entails that some attention is required to discuss the differentiability of the objective functional.  \par
It is worth to note that this kind of approach, based on the determination of a fundamental solution, could be used also for tackling the control problem associated with more canonical operators, as the Laplace operator or other elliptic operators. \par
Now we briefly summarize the content of this paper. In next section we introduce our notation and our basic geometrical assumptions. Section 3 is devoted to the mathematical analysis of the state equation: here, its solution is built up starting from the fundamental solution of a $\rot \rot$ operator. In Section 4, we present the optimal control problem and eventually derive first order optimality conditions. \par
To our best knowledge, this article represents the first contribution towards the optimal control of electromagnetic fields in the presence of spike sources.

\section{Preliminaries}

\phantom{axgfxfhc}
\par
\vspace{2mm}
\textbf{Geometrical assumptions.} The computational domain $\Omega$ is a bounded simply connected open set in $\mathbb{R}^3$ with Lipschitz boundary $\partial \Omega =: \Gamma$. A non-empty open, connected subset $\Omega_C \subset \Omega$ denotes the conducting region and consequently $\Omega_I := \Omega \setminus \overline{\Omega_C}$ is the insulator, which is also assumed to be connected for simplicity; $\Omega_C$ is strictly contained in $\Omega$ in such a way that $\Gamma \cap \partial \Omega_C = \emptyset$ and it is assumed to be simply connected, implying that $\Omega_I$ is also simply connected. The set $\Gamma_{C} := \partial \Omega_I \cap \partial \Omega_C$ is the interface between the conductor and the insulator. We finally set $\Gamma_I := \partial \Omega_I = \Gamma \cap \Gamma_C$ and denote by $\mathbf{n}, \mathbf{n}_C$ and $\mathbf{n}_I$ respectively the unit outward normal vectors on $\Gamma, \Gamma_C $ and $\Gamma_I$. From now on, for the sake of clarity we use the notation $\h_I := \h|_{\Omega_I}, \bm{\sigma}_C := \bm{\sigma}|_{\Omega_C} $ (and similar for
other fields) to explicitly underline to which subdomain a certain vector or matrix valued map is restricted. \par
\vspace{6mm}
\hspace{-6mm}\textbf{Notation.} Throughout this paper, we shall work with functional spaces on the field of complex numbers - unless otherwise specified - and we shall use a bold typeface to denote a three-dimensional vector map, or a vector space of three-dimensional vector functions. For instance, we set:
\begin{equation}
\begin{split}
    &\bm{L}^2(\Omega) := \{ \mathbf{u} : \Omega \rightarrow \mathbb{C}^3 \ | \ |\mathbf{u}| \in L_{\R}^2(\Omega)  \} \\
    &H^1(\Omega) := \{ u : \Omega \rightarrow \mathbb{C} \ | \ |u| \in L_{\R}^2(\Omega) , |\nabla u| \in L_{\R}^2(\Omega)  \};
    \end{split}
\end{equation}
the spaces $\bm{H}(\rot; \Omega), \bm{H}(\di ; \Omega)$ are thus defined as
\begin{equation*}
    \begin{split}
        &\bm{H}(\rot; \Omega) := \{ \mathbf{u} : \Omega \rightarrow \mathbb{C}^3 \ | \ \mathbf{u} \in \bm{L}^2(\Omega), \ \rot \mathbf{u} \in \bm{L}^2(\Omega)  \}, \\
        &\bm{H}(\di; \Omega) := \{ \mathbf{u} : \Omega \rightarrow \mathbb{C}^3 \ | \ \mathbf{u} \in \bm{L}^2(\Omega), \ \di \mathbf{u} \in L_{\mathbb{C}}^2(\Omega) \}.
    \end{split}
\end{equation*}
The corresponding trace spaces are defined, e.g., in Monk \cite[Section 3.5]{monk2003finite}  and Alonso Rodr\'iguez and Valli \cite[Appendix A.1]{Alo}. \par
The matrix-valued coefficients $\bm{\mu} \in L^{\infty}(\Omega ; \R^{3 \times 3}), \bm{\sigma} \in L^{\infty}(\Omega_C ; \R^{3 \times 3})$ and $\bm{\epsilon} \in L^{\infty}(\Omega_I ; \R^{3 \times 3})$ are all assumed to be symmetric and uniformly positive definite; moreover they satisfy an \textit{homogeneity condition} which is below introduced and motivated, see $\eqref{homogenetiy condition}$.    

\section{Analysis of the state equation: the eddy current problem with a dipole source}

Let us consider an $\e$-based formulation for the eddy current problem with a dipole source in the form of a Dirac mass, namely:
\begin{equation} \label{E based Eddy Current }
\left\{
\begin{aligned}
    \rot (\bm{\mu}^{-1} \rot \e ) + i \omega \bm{\sigma} \e &= - i \omega \mathbf{p} \delta_{\mathbf{x}_0} \qquad \text{in} \ \Omega   \\
\di(\bm{\epsilon}_I \e_I) &= 0 \qquad \textnormal{in} \ \Omega_I  \\
(\bm{\mu}^{-1} \rot \e_I ) \times \n &= \mathbf{0} \qquad \textnormal{on} \ \Gamma  \\
\bm{\epsilon}_I \e_I \cdot \n &= 0 \qquad \textnormal{on} \ \Gamma,
\end{aligned}
\right.
\end{equation}
where $\mathbf{p} \in \mathbb{R}^3, \omega > 0, \mathbf{x}_0 \in \Omega_C$ and $\delta_{\mathbf{x}_0}$ stands for the Dirac distribution centered at $\mathbf{x}_0$. We remind that equations $\eqref{E based Eddy Current }_{3,4}$ correspond to the choice of the so called \textit{magnetic boundary condition(s)}, see Alonso Rodr\'iguez and Valli \cite[Section 1.3]{Alo}, reinterpreted after eliminating the magnetic field from the eddy current system. \par
We also point out that $\eqref{E based Eddy Current }$ is somehow already a simplified model, since our geometrical assumptions entail that a couple of equations related to the topology of $\Omega_I$ can be a priori dropped, see again Alonso Rodr\'iguez and Valli \cite[p. 22]{Alo}.
\par
Prior to the control analysis, we need to address the well-posedness of problem $\eqref{E based Eddy Current }$. To this end, from now on we shall assume that physical parameters $\bm{\mu}, \bm{\sigma}$ satisfy a \textit{local homogeneity condition}: there exists a ball $B_r(\mathbf{x}_0)$ centered at $\mathbf{x}_0$ and two real positive constants $\mu_0, \sigma_0$ for which:
\begin{equation} \label{homogenetiy condition}
    \bm{\mu}(\mathbf{x}) = \mu_0\operatorname{Id}_{\R^3} \quad \textnormal{and } \ \bm{\sigma}(\mathbf{x}) = \sigma_0\operatorname{Id}_{\R^3} \qquad \forall \mathbf{x} \in B_{r}(\mathbf{x}_0).
\end{equation}
The latter assumption is not that much restrictive in most instances because the location of the point source, i.e. $\mathbf{x}_0$, is more or less free to choose and it seems reasonable opt for a point that does not lie on an interface region separating different materials. On the other hand, it is pivotal for giving a meaning to our fundamental solution-based approach: $\bm{\mu}$ being constant in a neighbourhood of $\mathbf{x}_0$ entails that locally we are dealing with the $\rot \rot - \operatorname{Id}$ operator - up to constants -, whose fundamental solution is known in the literature. The following result is adapted from Ammari \textit{et al.} \cite{bao2002inverseAmmari}: \par
\begin{prop} \label{fundamental solution prop}
Let $z = \sqrt{-i \omega \mu_0 \sigma_0}$ with $\operatorname{Re}z < 0$ and $\mathbf{q} = - i \omega \mathbf{p}$; the distributional solution $\mathbf{K} = \mathbf{K}(\cdot; \mathbf{x}_0)$ of the equation 
\begin{equation} \label{distributional equation}
    \rot \rot \mathbf{K} - z^2 \mathbf{K} = \mathbf{q} \delta_{\mathbf{x}_0} 
\end{equation}
is given by
\begin{equation} \label{fundamental solution K}
\begin{aligned}
    \mathbf{K} = \mathbf{K}(\mathbf{x}; \mathbf{x}_0) = \mathbf{q} \Phi_{\mathbf{x}_0}(\mathbf{x}) + \frac{1}{z^2} (\mathbf{q} \cdot \nabla) \nabla \Phi_{\mathbf{x}_0}(\mathbf{x}),
    \end{aligned}
\end{equation}
 where
 \begin{equation}
     \Phi_{\mathbf{x}_0} = \frac{\operatorname{exp}(iz|\mathbf{x}-\mathbf{x}_0|)}{4 \pi |\mathbf{x-\mathbf{x}_0}|}
 \end{equation}
  is the fundamental solution - up to translation in $\mathbf{x}_0$ - of the Helmholtz operator
  \begin{equation*}
      - \Delta - z^2 \operatorname{Id}. 
  \end{equation*}
\end{prop}
\begin{rmk}[Dependence of $\mathbf{K}$ on the intensity $\mathbf{q}$]
Since $\mathbf{q}$ is constant, we have
\begin{equation} \label{dependence of K on p}
\begin{split}
    \mathbf{K} &= 
 \mathbf{q} \Phi + \frac{1}{z^2} (\mathbf{q} \cdot \nabla) \nabla \Phi \\
    &= \operatorname{Id}(\mathbf{q} \Phi) + \frac{1}{z^2} ( \nabla^{2} \Phi ) \mathbf{q} \\
    &= [\operatorname{Id} \Phi + \nabla^2 \Phi] \mathbf{q} =: N \mathbf{q},
    \end{split} 
\end{equation}
where $N = N(\mathbf{x}_0, \Phi_{\mathbf{x}_0})$ is then a symmetric matrix with entries in $H^{-2}(\Omega)$, since it inherits the singularity of $\Phi_{\mathbf{x}_0}(\cdot)$ at $\mathbf{x}=\mathbf{x}_0$.
\end{rmk}
If $\mathbf{x} \in B_r(\mathbf{x}_0)$, equation $\eqref{E based Eddy Current }_1$ reads
\begin{equation*}
    \mu_0^{-1} \rot \rot \e(\mathbf{x}) + i \omega \sigma_0 \e(\mathbf{x}) = - i \omega \mathbf{p} \delta_{\mathbf{x}_0}(\mathbf{x}),
\end{equation*}
thus Proposition \ref{fundamental solution prop} applies and we are suggested to look for the solution of $\eqref{E based Eddy Current }$ in the form
\begin{equation} \label{split expression for E}
    \e = \mathbf{K} + \mathbf{M},
\end{equation}
where $\mathbf{M}$ has to read the behaviour outside the ball $B_r(\mathbf{x}_0)$ through a modified source on the RHS. More precisely, $\mathbf{M}$ is formally the solution to
\begin{equation} \label{system for Q}
    \left\{ 
    \begin{aligned}
    &\rot(\bm{\mu}^{-1} \rot \m ) + i \omega \bm{\sigma} \m = \J \quad \textnormal{in} \ \Omega \\
    &\di(\bm{\epsilon} \m)= - \di(\bm{\epsilon}\mathbf{K}) \quad \textnormal{in} \ \Omega_I \\
    &(\bm{\mu}^{-1} \rot \m) \times \n = - (\bm{\mu}^{-1} \rot \mathbf{K}) \times \n \quad \textnormal{on} \ \Gamma \\
    &\bm{\epsilon} \m \cdot \n = - \bm{\epsilon} \mathbf{K} \cdot \n \quad \textnormal{on} \ \Gamma,
    \end{aligned}
    \right.
\end{equation}
where
\begin{equation} \label{modified datum J}
\J=
    \left\{
    \begin{aligned}
        &\mathbf{0} \qquad &\textnormal{in }  B_r(\mathbf{x}_0) \\
        &- \rot(\bm{\mu}^{-1}\rot \mathbf{K}) - i \omega \bm{\sigma} \mathbf{K} \qquad &\textnormal{in } \Omega \setminus B_r(\mathbf{x}_0);
    \end{aligned}
    \right.
\end{equation}
later on we shall see the correct weak formulation of this formal problem. \par
Focusing now on $\eqref{system for Q}$,
we first aim at homogenising it by finding a vector field - in the form of a gradient - which has both the same divergence of $\bm{\epsilon} \mathbf{Q}$ in $\Omega_I$, and the same normal component on $\Gamma$. \par
\vspace{2mm}
Let
\begin{equation} \label{definition of W}
    W := \{ w \in H^1(\Omega_I) : w = 0 \ \textnormal{on } \Gamma_C \};
\end{equation}
$\eta_I \in W$ is defined as the solution of the weak boundary value problem
\begin{equation} \label{weak problem for Eta}
    b[\eta_I, \xi] := \int_{\Omega_I} \bm{\epsilon}_I \nabla \eta_I \cdot \nabla \overline{\xi} = - \int_{\Omega_I} \bm{\epsilon}_I \mathbf{K} \cdot \nabla \overline{\xi}  \qquad \forall \xi \in W  ,
\end{equation}
which is clearly well-posed since $\mathbf{K}|_{\Omega_I} \in \bm{L}^2(\Omega_I)$. It is straightforward to see that $\eta_I$ is the weak solution of the strong, mixed boundary value problem
\begin{equation} \label{strong problem for Eta}
    \left\{
    \begin{aligned}
    &-\di (\bm{\epsilon}_I \nabla \eta_I ) = \di (\bm{\epsilon}_I \mathbf{K}) \qquad \textnormal{in } \Omega_I \\
    &\eta_I = 0 \qquad \textnormal{on } \Gamma_C \\
    &\bm{\epsilon}_I \nabla \eta_I \cdot \n = - \bm{\epsilon}_I \mathbf{K} \cdot \n \qquad \textnormal{on } \Gamma.
    \end{aligned}
    \right.
\end{equation}
\par
We then extend $\eta_I$ by zero outside $\Omega_I$ and define 
\begin{equation}
    \eta := 
    \begin{cases}
    \eta_I \qquad \textnormal{in }\Omega_I \\
    0 \qquad \textnormal{in } \Omega_C
    \end{cases}
    \in H^1(\Omega).
\end{equation}

\begin{rmk}[Dependence of $\eta$ on $\mathbf{p}$] \label{first remark on dependence on P}
Since $\mathbf{q}= - i \omega \mu_0 \mathbf{p}$, the dependence with respect to $\p$ is given by 
\begin{equation} \label{definition of A and N}
    \mathbf{K} =- i \omega \mu_0 N \p =: A \mathbf{p},
\end{equation}
where $A=A(\mathbf{x}_0, \Phi_{\mathbf{x}_0})$
is defined as $A= - i \omega \mu_0 N$.
Since linearity is preserved by extensions to zero, the mapping $\R^3 \ni \mathbf{p} \mapsto \eta(\mathbf{p}) \in H^1(\Omega)$ is linear; in particular, the same is true for $\mathbf{p} \mapsto (\nabla \eta)(\mathbf{p})$.  
\end{rmk}
Back to problem $\eqref{system for Q}$, we can now split its solution as
\begin{equation} \label{split expression for M}
    \m = \q + \nabla \eta ,
\end{equation}
where $\q \in \bm{H}(\rot ; \Omega)$ has now to satisfy
\begin{equation} \label{system for Qstar}
    \left\{ 
    \begin{aligned}
    &\rot(\bm{\mu}^{-1} \rot \q ) + i \omega \bm{\sigma} \q = \J \quad \textnormal{in} \ \Omega \\
    &\di(\bm{\epsilon} \q)=0 \quad \textnormal{in} \ \Omega_I \\
    &(\bm{\mu}^{-1} \rot \q) \times \n = - (\bm{\mu}^{-1} \rot \mathbf{K}) \times \n \quad \textnormal{on} \ \Gamma \\
    &\bm{\epsilon} \q \cdot \n = 0 \quad \textnormal{on} \ \Gamma,
    \end{aligned}
    \right.
\end{equation}
$\J$ being still defined as in $\eqref{modified datum J}$. Notice that the singularity at $\mathbf{x}=\mathbf{x}_0$ of the initial problem $\eqref{E based Eddy Current }$ is directly \textit{read} by the fundamental solution $\mathbf{K}$ via $\eqref{split expression for E}$, hence we are left with a boundary value problem where $\mathbf{K}$ appears as a datum, but only in subsets of the domain where it is smooth.    \par
\vspace{2mm}
In order to set up a weak formulation of the eddy current problem $\eqref{system for Qstar}$, we introduce the linear space
\begin{equation} \label{definition of space V}
    \mathbf{V} := \{ \mathbf{v} \in \bm{H}(\rot; \Omega) : \di(\bm{\epsilon} \mathbf{v}_I)=0 \ \textnormal{in} \ \Omega_I, \ \bm{\epsilon} \mathbf{v} \cdot \n = 0 \ \textnormal{on} \ \Gamma \},
\end{equation}
which turns out to be a Hilbert space if endowed with the (semi)weighted inner product
\begin{equation} \label{inner product of H(curl)}
    \la \mathbf{u}, \mathbf{v} \ra_{\mathbf{V}} := \int_{\Omega} \bm{\epsilon} \mathbf{u} \cdot \overline{\mathbf{v}} + \int_{\Omega} \rot \mathbf{u} \cdot \rot \overline{\mathbf{v}}.
\end{equation}
Notice that the linear space $\mathbf{V}$ turns out to be suitable thanks to the preliminary homogenization by means of $\nabla \eta$. \par
Multiplying equation $\eqref{system for Qstar}_1$ by (the complex conjugate of) a test function $\mathbf{v} \in \mathbf{V}$, integrating in $\Omega$ and then by parts we obtain:
\begin{equation} \label{computations for Weak Formulation}
    \begin{split}
        \int_{\Omega} \J \cdot \overline{\mathbf{v}} &= \int_{\Omega} \bm{\mu}^{-1} \rot \q \cdot \rot \ov - \int_{\Gamma} [(\bm{\mu}^{-1} \rot \q  ) \times \n] \cdot \ov + i \omega \int_{\Omega_C} \bm{\sigma} \q \cdot \ov \\
        &= \int_{\Omega} \bm{\mu}^{-1} \rot \q \cdot \rot \ov + i \omega \int_{\Omega_C} \bm{\sigma} \q \cdot \ov + \int_{\Gamma} [(\bm{\mu}^{-1} \rot \mathbf{K}  ) \times \n] \cdot \ov.
    \end{split}
\end{equation}
 It is important to point out that the boundary integrals shall be generally understood as duality pairings between $(\bm{\mu}^{-1} \rot \mathbf{K} \times \n) \in \bm{H}^{-1/2}(\di_{\tau}, \Gamma)$ and $\n \times \overline{\mathbf{v}} \times \n \in \bm{H}^{-1/2}(\rot_{\tau}; \Gamma)$.
Let us rigorously see what is $\displaystyle \int_{\Omega} \J \cdot \ov$ for the datum $\J$ defined in $\eqref{modified datum J}$. Let $B^c_{\mathbf{x}_0} := \Omega \setminus B_r(\mathbf{x}_0)$; by the homogeneity condition $\eqref{homogenetiy condition}$ we can write:
\begin{equation*}
    \begin{split}
        \int_{\Omega} \J \cdot \ov &= \int_{B^c_{\mathbf{x}_0}} [ - \rot( \bm{\mu}^{-1} \rot \mathbf{K}) - i \omega \bm{\sigma} \mathbf{K} ] \cdot \ov \\
        &= \int_{\Omega} [ - \rot( \bm{\mu}^{-1} - \mu_0^{-1}  ) \rot \mathbf{K}) \cdot \ov - i \omega (\bm{\sigma} - \sigma_0) \mathbf{K} \cdot \ov ] \\
        &= \int_{\Omega} [  - (\bm{\mu}^{-1} - \mu_0^{-1}) \rot \mathbf{K} \cdot \rot \ov - i \omega (\bm{\sigma} - \sigma_0) \mathbf{K} \cdot \ov] \\
        &\phantom{12345}-\int_{\Gamma} \n \times [(\bm{\mu}^{-1} - \mu_0^{-1})\rot  \mathbf{K}] \cdot \ov \\
        &=\int_{\Omega} [  - (\bm{\mu}^{-1} - \mu_0^{-1}) \rot \mathbf{K} \cdot \rot \ov - i \omega (\bm{\sigma} - \sigma_0) \mathbf{K} \cdot \ov] + \int_{\Gamma} (\n \times \mu_0^{-1} \rot \mathbf{K}) \cdot \ov \\
        &\phantom{12234}-\int_{\Gamma} (\n \times \bm{\mu}^{-1} \rot \mathbf{K}) \cdot \ov,
    \end{split}
\end{equation*}
where with a slight abuse of notation $\mu_0^{-1}$ has been used in place of $\mu_0^{-1}\operatorname{Id}$. Combining the above computation with $\eqref{computations for Weak Formulation}$ and $\eqref{system for Qstar}_3$, we conclude that the weak formulation of $\eqref{system for Qstar}$ reads as follows:
\begin{prob} \label{problem one}
Let $\mathbf{K}$ be defined in $\eqref{fundamental solution K}$. To find $\q \in \mathbf{V}$ such that
\begin{equation} \label{weak formulation for Qstar}
\begin{split}
    a^{+}[\mathbf{\q}, \mathbf{v}] &:= \int_{\Omega} \bm{\mu}^{-1} \rot \q \cdot \rot \ov + i \omega \int_{\Omega_C} \bm{\sigma} \q_C \cdot \ov \\ &= \int_{B^c_{\mathbf{x}_0}} [  - (\bm{\mu}^{-1} - \mu_0^{-1}) \rot \mathbf{K} \cdot \rot \ov - i \omega (\bm{\sigma} - \sigma_0) \mathbf{K} \cdot \ov] + \int_{\Gamma} (\n \times \mu_0^{-1} \rot \mathbf{K}) \cdot \ov,
    \end{split}
\end{equation}
for all $\mathbf{v} \in \mathbf{V}$.
\end{prob}
\par
The following Poincarè-type inequality (see Alonso Rodr\'iguez and Valli \cite[Lemma 2.1]{Alo}, Fernandes and Gilardi \cite{fernandes1997PoincareForCurl}) will be pivotal to prove the well-posedness of Problem \ref{problem one}. 
\begin{lem} \label{Poincare Curl}
There is a constant $C_0>0$ such that
\begin{equation}
    \| \mathbf{w}_I \|_{0,\Omega_I} \le C_0 (\|\rot \mathbf{w}_I \|_{0,\Omega_I} + \| \di(\bm{\epsilon_I} \mathbf{w}_I) \|_{0,\Omega_I} + \| \mathbf{w}_I \times \n_I \|_{-1/2, \di_{\tau, \Gamma_C}} + \| \bm{\epsilon} \mathbf{w}_I \cdot \n \|_{-1/2,\Gamma} )
\end{equation}
for all $\mathbf{w}_I \in \bm{H}(\rot; \Omega_I) \cap \bm{H}_{\bm{\epsilon}_I}(\di; \Omega_I)$ with $\mathbf{w}_I \perp^{\bm{\epsilon}_I} \mathcal{H}_{\bm{\epsilon}_I}(\Gamma_C, \Gamma; \Omega_I)$, where
\begin{equation}
\begin{split}
     \mathcal{H}_{\bm{\epsilon}_I}(\Gamma_C, \Gamma; \Omega_I) = \{ \mathbf{q}_I \in \bm{L}^2(\Omega_I) : \rot \mathbf{q}_I = \mathbf{0}, \di(\bm{\epsilon}_I \mathbf{q}_I) = 0, \\ \mathbf{q}_I \times \n_I= \mathbf{0} \ \textnormal{on} \ \Gamma_C, \ \bm{\epsilon}_I \mathbf{q}_I \cdot \n_I = 0 \ \textnormal{on} \ \Gamma \},
\end{split}
\end{equation}
and $\perp^{\bm{\epsilon}_I}$ denotes the orthogonality with respect to the $\bm{\epsilon}_I$-weighted $\bm{L}^2(\Omega_I)$ inner product, that is $(\bm{\epsilon}_I \cdot , \cdot)_{\bm{L}^2(\Omega_I)}$.
\end{lem}
We shall briefly explain why the above lemma actually applies to functions in $\mathbf{V}$. Indeed first of all $\mathbf{V} \subset \bm{H}(\rot; \Omega_I) \cap \bm{H}_{\bm{\epsilon}_I}(\di; \Omega_I)$ due to the divergence-free constraint. Moreover, the space $\mathcal{H}_{\bm{\epsilon}_I}(\Gamma_C, \Gamma; \Omega_I)$ has dimension equal to $p_{\Gamma_C} + n_{\Gamma}$, where the former denotes the number of connected components of $\Gamma_C$ minus one, while the latter denotes the number of $\Gamma$-independent non-bounding cycles\footnote{More precisely, we say that a family $\mathcal{C}$ of disjoint cycles of $\Omega_I$ is formed by $\Gamma$-independent, non-bounding cycles if, for each non trivial subfamily $\mathcal{C}^{*} \subset \mathcal{C}$, the union of the cycles in $\mathcal{C}^{*}$ cannot be equal to $S \setminus \gamma$, where $S$ denotes a surface contained in $\Omega_I$ and $\gamma$ a union of cycles contained in $\Gamma$.} in $\Omega_I$, and both these numbers vanish under the hypothesis that $\Gamma_C$ is connected and $\Omega$ is simply connected\footnote{The fact that the computational domain $\Omega$ is simply connected is sufficient to make $n_{\Gamma}$ equal to zero. However, this may also happen when the topology of $\Omega$ is non-trivial. For a detailed discussion and examples we refer to Alonso Rodr\'iguez and Valli \cite[Section 1.4]{Alo}}. \par
\vspace{3mm}
A consequence of the previous lemma is the following:
\begin{cor} \label{coercivity}
The sesquilinear forms
\begin{equation} \label{sesquilinear forms for Qstar}
\begin{split}
    &a^{+}[\mathbf{w}, \mathbf{v}] = \int_{\Omega} \bm{\mu}^{-1} \rot \mathbf{w} \cdot \rot \ov + i \omega \int_{\Omega_C} \bm{\sigma} \mathbf{w} \cdot \ov, \\
    &a^{-}[\mathbf{w}, \mathbf{v}] := \int_{\Omega} \bm{\mu}^{-1} \rot \mathbf{w} \cdot \rot \ov - i \omega \int_{\Omega_C} \bm{\sigma} \mathbf{w} \cdot \ov
    \end{split}
\end{equation}
are (strongly) coercive in $\mathbf{V} \times \mathbf{V} $. 
\end{cor}
\begin{proof}
For all $\mathbf{v} \in \mathbf{V}$, we have:
\begin{equation*}
\setlength{\jot}{10pt}
\begin{split}
    |a^{+}[\mathbf{v}, \mathbf{v}]|^2 &= \left( \int_{\Omega} \bm{\mu}^{-1} \rot \mathbf{v} \cdot \rot \overline{\mathbf{v}} \right)^2 + \omega^2 \left( \int_{\Omega_C} \bm{\sigma} \mathbf{v}_C \cdot \overline{\mathbf{v}}_C \right)^2 \\
    &\ge \left\{ \mu_{\textnormal{min}}^{-2} \| \rot \mathbf{v} \|_{0, \Omega}^4 + \omega^2 \sigma_{\textnormal{min}}^{-2} \| \mathbf{v}_C \|_{0, \Omega_C}^4 )  \right\} \\
    &\ge C(\| \rot \mathbf{v} \|^2_{0, \Omega} + \| \mathbf{v}_C \|^2_{0, \Omega_C})^2. 
    \end{split}
    \end{equation*}
By Lemma \ref{Poincare Curl} together with the continuity of the tangential trace, we also have:
\begin{equation*}
    \begin{split}
        \| \mathbf{v} \|_{0, \Omega_I}^2 &\le C_0 (\| \rot \mathbf{v}_I \|_{0, \Omega_I} + \| \mathbf{v}_I \times \n_I  \|_{-1/2, \di_{\tau}, \Gamma_C})^2 \\
        &= C_0( \| \rot \mathbf{v}_I \|_{0, \Omega_I} + \| \mathbf{v}_C \times \n_C  \|_{-1/2, \di_{\tau}, \Gamma_C})^2 \\
        &\le C_1( \| \rot \mathbf{v}_I \|^2_{0, \Omega_I} + \| \mathbf{v}_C  \|_{0, \Omega_C}^2 + \| \rot \mathbf{v}_C \|_{0, \Omega_C}^2 ). 
    \end{split}
\end{equation*}
    Therefore
    \begin{equation*}
        \begin{split}
            |a^{+}[\mathbf{v}, \mathbf{v}]|^2 &\ge C_2 (\| \rot \mathbf{v} \|^2_{0, \Omega} + \| \mathbf{v}_C \|^2_{0, \Omega_C} + \| \mathbf{v}_I \|^2_{0, \Omega_I} )^2 \\
            &=C_2 \| \mathbf{v} \|^4_{\mathbf{V}},
        \end{split}
    \end{equation*}
    $C_0, C_1, C_2$ being positive real constants, which do not depend on $\mathbf{v}$.
 Since $a^{+}[\cdot, \cdot], a^{-}[\cdot, \cdot]$ have the same magnitude, the proof is complete.
\end{proof}
For the sake of completeness, we shall briefly discuss how to proceed when $\Gamma_C$ is \textit{not} assumed to be connected\footnote{Since $\Omega_I$ is assumed to be connected, this can only happen if $\Omega_C$ is a non-connected conductor, that is $\displaystyle \Omega_C = \coprod_i \Omega_C^{(i)}$ with $\Omega_C^{(i)}$ connected for each $i$. The presence of more conductors in a device is a situation that often arises in engineering applications.}; in this case $p_{\Gamma_C} \ge 1$, then it is known (Alonso Rodr\'iguez and Valli \cite[Appendix A.4]{Alo}) that $\mathcal{H}_{\bm{\epsilon}_I}(\Gamma_C, \Gamma; \Omega_I)$ is spanned by $ \{ \nabla w_i \}_{i=1...p_{\Gamma_C}}$, $w_i \in H^1(\Omega_I)$ being the solution of the mixed problem:
\begin{equation}
\left\{
    \begin{aligned}
        &\di(\bm{\epsilon}_I \nabla w_i) = 0 \qquad \textnormal{in } \Omega_I \\
        &\bm{\epsilon}_I \nabla w_i \cdot \n = 0 \qquad \textnormal{on } \Gamma \\
        &w_i = 0 \qquad \textnormal{on } \Gamma_C \setminus \Gamma_i \\
        &w_i = 1 \qquad \textnormal{on } \Gamma_i,
        \end{aligned}
    \right.
\end{equation}
where $(\Gamma_i)_{i=1...p_{\Gamma_C}}$ denotes the \textit{i}-th connected component. Fix any $j \in \{ 1 \dots p_{\Gamma_C} \}$; for each $\mathbf{v} \in \mathbf{V}$, we have:
\begin{equation*}
\begin{split}
    \int_{\Omega_I} \bm{\epsilon}_I \mathbf{v}_I \cdot \nabla w_j = - \int_{\Omega_I} w_j \di(\bm{\epsilon}_I \mathbf{v}_I) + \int_{\partial \Omega_I} w_j \bm{\epsilon}_I \mathbf{v}_I \cdot \n &= \int_{\Gamma} w_j \bm{\epsilon}_I \mathbf{v}_I \cdot \n + \sum_{i=1}^{p_{\Gamma_C}} \int_{\Gamma_i} w_j \bm{\epsilon}_I \mathbf{v}_I \cdot \n  \\
    &=\int_{\Gamma_j} \bm{\epsilon}_I \mathbf{v}_I \cdot \n, \\
    \end{split} 
\end{equation*}
since $\bm{\epsilon}_I \mathbf{v}_I \cdot \n$ vanishes identically on the external boundary $\Gamma$. Hence we see that it suffices to require the functions of $\mathbf{V}$ to satisfy the additional constraints
\begin{equation*}
    \int_{\Gamma_i} \bm{\epsilon}_I \mathbf{v}_I \cdot \n = 0 \qquad \forall i = 1 \dots p_{\Gamma_C}
\end{equation*}
concerning the fluxes through each connected component of the boundary of the conductor, $\Gamma_C$.\par
With this adjustment, $\mathbf{V}$ is yet again a Hilbert space endowed with the $\bm{H}(\rot;\Omega)$ inner product $\eqref{inner product of H(curl)}$ and its elements satisfy the orthogonality hypothesis of Lemma \ref{Poincare Curl}, which, in turn, implies that Corollary \ref{coercivity} and the following lemma still hold. \par
\vspace{2mm}
\begin{lem}[Existence for $\q$] \label{Existence of Qstar}
Problem \ref{problem one} has a unique solution $\q \in \mathbf{V}$.
\end{lem}
\begin{proof}
The mapping $L : \bm{H}(\rot; \Omega) \mapsto \mathbb{C}$ defined via
\begin{equation} \label{Psi v,p}
\begin{split}
    L(\mathbf{v}) := \int_{B^c_{\mathbf{x}_0}} [  - (\bm{\mu}^{-1} - \mu_0^{-1}) \rot \mathbf{K} \cdot \rot \mathbf{v} - i \omega (\bm{\sigma} - \sigma_0) \mathbf{K} \cdot \mathbf{v}] + \int_{\Gamma} (\n \times \mu_0^{-1} \rot \mathbf{K}) \cdot \mathbf{v}
    \end{split}
\end{equation}
(that is, the complex conjugate of the right hand side of $\eqref{weak formulation for Qstar}$) is linear and continuous on $\mathbf{V}$ owing again to the continuity of the tangential trace; moreover by Corollary \ref{coercivity} the sesquilinear form $a[\cdot, \cdot]$ is coercive on $\mathbf{V} \times \mathbf{V}$, hence the Lax-Milgram lemma applies ensuring the existence of a unique weak solution to $\eqref{weak formulation for Qstar}$.
\end{proof}
Summarizing the whole discussion on the state equation, we end up with:
\begin{teo}[Well-posedness for state equation] \label{well posedness of state eq}
Assuming that condition $\eqref{homogenetiy condition}$ is satisfied, there exists a solution $\mathbf{E} \in \bm{H}^{-2}(\Omega)$ to $\eqref{E based Eddy Current }$, which can be written as:
\begin{equation} \label{decomposition of E}
    \e = \q + \nabla \eta + \mathbf{K},
\end{equation}
where $\q$ is the solution of $\eqref{system for Qstar}$, $\eta$ is the solution of $\eqref{weak problem for Eta}$ and $\mathbf{K}$ is the fundamental solution defined in $\eqref{fundamental solution K}$. Moreover, it is unique among all solutions $\widehat{\e}$ such that $(\widehat{\e}-\mathbf{K}) \in \bm{H}(\rot ; \Omega)$.
\end{teo}
\begin{proof}
Uniqueness is the only assertion yet to be proved. Assume that $\widehat{\e}$ is another solution for which $(\widehat{\e} - \mathbf{K}) \in \bm{H}(\rot ; \Omega)$, we can write it as $\widehat{\e}= \mathbf{K} + (\widehat{\e} - \mathbf{K})$ and it is easy to see that the addendum $\widehat{\e} - \mathbf{K}$ is a solution to $\eqref{system for Q}$, a problem for which one has uniqueness in $\bm{H}(\rot ; \Omega)$. Hence we conclude $\widehat{\e} - \mathbf{K} = \e - \mathbf{K}$ and $\e = \widehat{\e}$.
\end{proof}
\begin{cor}[Linearity in $\mathbf{p}$] \label{linearity in p of the solution operator}
The solution mapping $S: \R^3 \rightarrow \bm{H}^{-2}(\Omega)$ acting as
\begin{equation} \label{control to state mapping}
    \mathbf{p} \mapsto S\mathbf{p} := \e(\mathbf{p}), \ \textnormal{with } (\e(\mathbf{p})- \mathbf{K}_{\mathbf{p}}) \in \bm{H}(\rot; \Omega)
\end{equation}
is linear (with respect to real numbers).
\end{cor}
\begin{proof}
We see that each term on the RHS of $\eqref{decomposition of E}$ is linear in $\mathbf{p}$. Indeed remark \ref{first remark on dependence on P} is enough for $\mathbf{K}, \nabla \eta$; for what concerns $\q$, it suffices to observe that the mapping $\mathbf{V} \ni \mathbf{v} \mapsto L_{\mathbf{p}}(\mathbf{v})$ defined via $\eqref{Psi v,p}$ depends linearly on $\mathbf{p}$, that is $L_{\alpha \mathbf{p}_1 + \beta \mathbf{p}_2} = \alpha L_{\mathbf{p}_1} + \beta L_{\mathbf{p}_2}$ as elements of $\mathcal{L}(\mathbf{V}; \mathbb{C})$, with $\alpha, \beta \in \mathbb{R}$ and $\mathbf{p}_1, \mathbf{p}_2 \in \mathbb{R}^3$.
\end{proof}
\begin{rmk}[Other boundary conditions]
The boundary conditions $\eqref{E based Eddy Current }_{3,4}$ are not the most commonly seen for an $\e$-based eddy current system. Let us briefly state what changes if the so called \textit{electric boundary condition}
\begin{equation} \label{electric boundary condition}
\e_I \times \n = \mathbf{0} \qquad \textnormal{on } \Gamma
\end{equation}
is considered in place of $\eqref{E based Eddy Current }_{3.4}$. Again we look for a solution in the form $\e= \mathbf{K} + \mathbf{M} + \nabla \eta$ (see $\eqref{split expression for E}$ together with $\eqref{split expression for M}$), exception made for the fact that now $\eta|_{\partial \Omega_I} = 0$ instead of $\eqref{strong problem for Eta}_3$ and we are left with the following formal problem:
\begin{equation} \label{system for Q remark}
    \left\{ 
    \begin{aligned}
    &\rot(\bm{\mu}^{-1} \rot \q ) + i \omega \bm{\sigma} \q = \J \quad \textnormal{in} \ \Omega \\
    &\di(\bm{\epsilon} \q)=0 \quad \textnormal{in} \ \Omega_I \\
    &\q \times \n = - \mathbf{K} \times \n =: \mathbf{G} \quad \textnormal{on} \ \Gamma.
    \end{aligned}
    \right.
\end{equation}
We set
\begin{equation*}
    \mathbf{V}_0 := \{ \mathbf{u} \in \bm{H}(\rot; \Omega) : \di(\bm{\epsilon}_I \mathbf{u}_I) = 0 \textnormal{ in }\Omega_I, \ \mathbf{u}_I \times \n = \mathbf{0} \textnormal{ on }\Gamma \};
\end{equation*}
since the bilinear form $a^{+}[\cdot, \cdot]$ is coercive in $\mathbf{V}_0$ (Lemma \ref{Poincare Curl}, and thus Corollary \ref{coercivity}, applies to functions of $\mathbf{V}_0$ too), the resolution procedure becomes standard if we are able to find a suitable\footnote{Note that $\widetilde{\mathbf{G}} \in \bm{H}(\rot; \Omega)$ would not be enough: if, say, $\mathbf{Q}_0$ solves the problem with homogeneous boundary datum $\eqref{system for Q remark}_3$ and $\widetilde{\mathbf{G}} \in \bm{H}(\rot; \Omega)$, then $\mathbf{Q} = \mathbf{Q}_0 + \widetilde{\mathbf{G}}$ does not need to satisfy the divergence-free constraint $\eqref{system for Q remark}_2$, although it satisfies the boundary condition $\eqref{system for Q remark}_3$.} lifting $\widetilde{\mathbf{G}}$ of $\mathbf{G}$, that is $\widetilde{\mathbf{G}} \in \mathbf{V}_0$ and $\widetilde{\mathbf{G}} \times \n = \mathbf{G}$ on $\Gamma$. \par
Let us consider the following $\rot-\di$ system for $\widetilde{\mathbf{G}}_I\in \bm{H}(\rot; \Omega_I)$ :
\begin{equation} \label{lifting of G}
    \left\{
    \begin{aligned}
        &\rot \widetilde{\mathbf{G}}_I = \bm{\Psi} \qquad \textnormal{in } \Omega_I \\
        &\di(\bm{\epsilon}_I \widetilde{\mathbf{G}}_I)=0 \qquad \textnormal{in } \Omega_I \\
        &\widetilde{\mathbf{G}}_I \times \n = \mathbf{G} \qquad \textnormal{on } \Gamma \\
        &\widetilde{\mathbf{G}}_I \times \n = \mathbf{0} \qquad \textnormal{on } \Gamma_C \\
        &\int_{\Gamma} \widetilde{\mathbf{G}}_I \cdot \n = 0,
    \end{aligned}
    \right.
\end{equation}
where $\bm{\Psi} = \nabla \phi$ and $\phi \in H^1(\Omega_I)$ satisfies
\begin{equation} \label{Neumann for electric boundary}
    \left\{
    \begin{aligned}
        &\Delta \phi = 0 \qquad \textnormal{in } \Omega_I \\
        &\nabla \phi \cdot \n = 0 \qquad \textnormal{on } \Gamma_C \\
        &\nabla \phi \cdot \n = \di_{\tau} \mathbf{G} \qquad \textnormal{on } \Gamma \\
        &\int_{\Omega_I} \phi = 0.
    \end{aligned}
    \right.
\end{equation}
In this way, we see that all compatibility conditions for the solvability of the $\rot-\di$ system (we refer to Alonso Rodr{\'i}guez \textit{et al.} in \cite[Chap. 1, Section 2.1]{Maxwell2019}) are satisfied. In particular, they are also sufficient for existence and uniqueness. \par
Indeed the Neumann problem $\eqref{Neumann for electric boundary}$ is well-posed since $\int_{\Gamma} \di_{\tau} \mathbf{G} = -\int_{\Gamma} \mathbf{G} \cdot (\nabla_{\tau} 1) = 0$, while for $\eqref{lifting of G}$ we have $\displaystyle \di \bm{\Psi} = \di \nabla \phi =0$ in $\Omega_I$ and $\di_{\tau} \mathbf{G} = \nabla \phi \cdot \n = \bm{\Psi} \cdot \n$ on $\Gamma$ by construction. Moreover the space of harmonic fields
\begin{equation*}
    \mathcal{H}(m; \Omega_I) := \{ \bm{\rho} \in \bm{L}^2(\Omega_I) : \rot \bm{\rho} = \mathbf{0} \textnormal{ in } \Omega_I, \di(\bm{\rho})=0 \textnormal{ in }\Omega_I, \bm{\rho} \cdot \n = 0 \textnormal{ on } \partial \Omega_I \}
\end{equation*}
is trivial since $\Omega_I$ is simply connected (see Alonso Rodr{\'i}guez and Valli \cite[Appendix A.4]{Alo}). Hence $\eqref{lifting of G}$ has a unique solution and eventually we can define 
\begin{equation*}
    \widetilde{\mathbf{G}} :=
    \left\{
    \begin{aligned}
        &\widetilde{\mathbf{G}}_I \qquad \textnormal{in } \Omega_I \\
        &\mathbf{0} \qquad \textnormal{in } \Omega_C
    \end{aligned}
    \in \mathbf{V}_0 \subset \bm{H}(\rot; \Omega),
    \right.
\end{equation*}
which is the desired lifting.
\end{rmk}

\section{The control problem}

Let us now discuss the optimal control problem; our analysis will be driven by the following task: suppose we want to approach two given desired electric field
(state functions) $\e_{d}, \h_d \in \bm{L}^2(\Omega)$ controlling the dipole intensity $\mathbf{p} \in \mathbb{R}^3$ (its location has already been fixed in $\mathbf{x}_0$, see $\eqref{E based Eddy Current }$); since the solution $\e$ to $\eqref{E based Eddy Current }$ does not belong to $\bm{L}^2(\Omega)$ due to the singularity at $\mathbf{x}=\mathbf{x}_0$ of the fundamental solution $\mathbf{K}$, we shall optimize the distance between the solution and the desired fields with respect to $\bm{L}^2(B_{\mathbf{x}_0^{c}})$, where $B_{\mathbf{x}_0}^{c} = \Omega \setminus \overline{B_{r}(\mathbf{x_0})}$ (the radius $r$ has already been chosen prior to the homogeneity assumption $\eqref{homogenetiy condition})$. In other words, although the eddy current state equation is driven by a (Dirac) dipole source concentrated at $\mathbf{x}_0$, the optimization problem disregards the behaviour of the state variable around (close to) the point $\mathbf{x}_0$. This may seem to be unreasonable at first sight, however, our resolution approach guarantees a priori the presence of a singularity of the same kind of $\mathbf{K}$ at $\mathbf{x}=\mathbf{x}_0$ and therefore we precisely focus the attention on the the state variable away from that point. In other words, we shall not be interested in a specific \textit{shape} at the actuators, we aim at given fields in the complement of the actuators instead. This kind of approach is often seen in optimal control problems for PDEs where a control domain $\Omega_{ctr}$ and a disjoint state observation domain $\Omega_{o}$ are considered, see e.g. Clason and Kunisch \cite{Clason:2012:MSA:2385297.2385333} or Pieper and Vexler \cite{pieper2013priori}. In this sense, here we are doing something similar taking $\Omega_{o} := B^c_{\mathbf{x}_0}$ and $\Omega_{ctr} := \{ \mathbf{x}_0 \}$.
\par 
 Summing up, we are then led to the following regularized problem:
\begin{equation}
\label{completecostfunctional}
\min_{\mathbf{p} \in \mathcal{P}_{ad}} F(\e, \mathbf{p}) := \frac{\nu_{E}}{2} \int_{B^{c}_{\mathbf{x}_0}} | \e - \e_{d} |^2 + \frac{\nu_H}{2} \int_{B^{c}_{\mathbf{x}_0}} |\bm{\mu}^{-1}\rot \e - \h_{d} |^2 + \frac{\nu}{2} |\mathbf{p}|^2_{\mathbb{R}^3}, \\ 
\end{equation}
subject to 
\begin{empheq}[box=\fbox]{align} 
\rot (\bm{\mu}^{-1} \rot \e ) + i \omega \bm{\sigma} \e &= - i \omega \mathbf{p} \delta_{\mathbf{x}_0} \qquad \text{in} \ \Omega  \label{maxwellstate1} \\
\di(\bm{\epsilon}_I \e_I) &= 0 \qquad \textnormal{in} \ \Omega_I \label{maxwellstate2} \\
(\bm{\mu}^{-1} \rot \e_I ) \times \n &= \mathbf{0} \qquad \textnormal{on} \ \Gamma \label{maxwellstate3} \\
\bm{\epsilon}_I \e_I \cdot \n &= 0 \qquad \textnormal{on} \ \Gamma \label{maxwellstate4},
\end{empheq}
where 
\begin{equation*}
    \mathcal{P}_{ad} := \{ \mathbf{p} \in \R^3 : |(\mathbf{p})_i| \le p_{max}, \ i=1 \dots 3 \},
\end{equation*}
$0 < p_{max}$ being a bound for the maximal component-wise dipole intensity. \par
The fact that $\mathbf{K}$ is smooth far from $\mathbf{x}_0$ together with the assumption that $\e_d, \h_d \in \bm{L}^2(\Omega)$ ensure that both $(\e - \e_d)$ and $(\mu^{-1} \rot \e - \h_d)$ lie in $\bm{L}^2(B^c_{\mathbf{x}_0})$, making $F$ well-defined on $\bm{H}^{-2}(\Omega) \times \mathcal{P}_{ad}$. \par
Before proceeding further, we define the following reduced cost functional by composition with the control-to-state mapping $\eqref{control to state mapping}$:
\begin{equation}
\begin{split} \label{reduced cost funcional}
    F(\p) &:= \frac{\nu_E}{2} \| S \p - \e_d \|_{0,B^c_{\mathbf{x}_0}}^2 + \frac{\nu_H}{2} \| \bm{\mu}^{-1} \rot (S \p) - \h_d \|_{0,B^c_{\mathbf{x}_0}}^2 + \frac{\nu}{2} |\p|_{\R^3}^2 \\
    &= \frac{\nu_E}{2} \| \e_{\p} - \e_d \|_{0,B^c_{\mathbf{x}_0}}^2 + \frac{\nu_H}{2} \| \bm{\mu}^{-1} \rot \e_{\p} - \h_d \|_{0,B^c_{\mathbf{x}_0}}^2 + \frac{\nu}{2} |\p|_{\R^3}^2 ;
    \end{split}
\end{equation}
if $\nu > 0$, thanks to the continuity of $S$ we obtain at once that $F$ is weakly lower semi-continuous and strictly convex. This together with the fact that $\mathcal{P}_{ad}$ is compact entails by standard arguments (see Tr{\"o}ltzsch \cite[Section 2.5]{troltzsch2010optimal}) the existence and uniqueness of an optimal control $\mathbf{p}^{*} \in \mathcal{P}_{ad}$
such that
\begin{equation*}
    F(\mathbf{p}^{*}) = \min_{\mathbf{p} \in \mathcal{P}_{ad}} F(\mathbf{p});
\end{equation*}
 with this optimal control an optimal state $\e^{*} = S \mathbf{p}^{*} \in \bm{H}^{-2}(\Omega)$ is associated. If $\nu = 0$, we still have existence but uniqueness is no longer guaranteed.

\subsection{Necessary and sufficient conditions for optimality}

By theorem $\eqref{well posedness of state eq}$, we know that to each control $\mathbf{p} \in \mathcal{P}_{ad}$ there corresponds a unique state
\begin{equation} \label{split structure of the state}
    \e_{\mathbf{p}} = \q_{\mathbf{p}} + \nabla \eta_{\mathbf{p}} + \mathbf{K}_{\mathbf{p}};
\end{equation}
prior to deriving and discussing necessary (and sufficient) conditions for optimality, we need to further clarify Corollary \ref{linearity in p of the solution operator} on the dependence of $\mathbf{E}$ on $\mathbf{p}$, in particular the one of $\mathbf{Q}, \eta$ on $\mathbf{p}$. \par
We shall verify that the whole RHS of $\eqref{weak problem for Eta}$ depends linearly (at least w.r.t real numbers) on the control $\mathbf{p}$: this will be pivotal for deriving optimality conditions with an effective notation. We then perform a similar computation for the RHS of problem $\eqref{weak formulation for Qstar}$. For $\eqref{weak problem for Eta}$ we have:
     \begin{equation} 
         \begin{split}
     - \int_{\Omega_I} \bm{\epsilon}_I \mathbf{K}_{\p} \cdot \nabla \overline{\xi} = - \int_{\Omega_I} \bm{\epsilon}_I A \mathbf{p} \cdot \nabla \overline{\xi} = - \int_{\Omega_I}  \mathbf{p} \cdot A^T ( \bm{\epsilon}_I \nabla \overline{\xi} ) = \p \cdot \left( \int_{\Omega_I} - A^T (\bm{\epsilon}_I \nabla \overline{\xi}) \right),
         \end{split}
     \end{equation}
     and we thus define
     \begin{equation} \label{widetilde Mathcal G}
         \widetilde{\mathcal{G}}(\xi) := \int_{\Omega_I} - A^T (\bm{\epsilon}_I \nabla \overline{\xi}), \qquad \xi \in W. 
     \end{equation}
     Instead for $\eqref{weak formulation for Qstar}$ we obtain
     \begin{equation} 
         \begin{split}
             \int_{B^c_{\mathbf{x}_0}} & [  - (\bm{\mu}^{-1} - \mu_0^{-1}) \rot \mathbf{K}_{\p} \cdot \rot \ov - i \omega (\bm{\sigma} - \sigma_0) \mathbf{K}_{\p} \cdot \ov] \\
             &+ \int_{\Gamma} (\n \times \mu_0^{-1} \rot \mathbf{K}_{\p}) \cdot \ov \\
             &= \int_{B^c_{\mathbf{x}_0}} [  - (\bm{\mu}^{-1} - \mu_0^{-1}) \rot (A \p) \cdot \rot \ov - i \omega (\bm{\sigma} - \sigma_0) A \p \cdot \ov]\\
             &+ \int_{\Gamma} (\n \times \mu_0^{-1} \rot(A\p)) \cdot \ov \\
            &= \int_{B^c_{\mathbf{x}_0}}  [  - (\bm{\mu}^{-1} - \mu_0^{-1}) \sum_{j=1}^3 \rot A^{(j)} p_j  \cdot \rot \ov - i \omega \p \cdot A^T(\bm{\sigma} - \sigma_0)  \ov] \\ &+ \int_{\Gamma} \n \times \mu_0^{-1} \sum_{j=1}^3 \rot A^{(j)} p_j \cdot \ov \\
             &=  \sum_{j=1}^3 p_j \left(  - \int_{B^c_{\mathbf{x}_0}} [   (\bm{\mu}^{-1} - \mu_0^{-1})\rot A^{(j)} \cdot \rot \ov - \int_{B^c_{\mathbf{x}_0}} i \omega \p \cdot A^{(j)} [(\bm{\sigma} - \sigma_0)  \ov] \right. \\
             &+ \left. \int_{\Gamma} [\n \times \mu_0^{-1} \rot A^{(j)}] \cdot \ov \right),
         \end{split}
     \end{equation}
     and we define the vector $\displaystyle \mathcal{G}(\mathbf{v})$ component-wise via
     \begin{equation} \label{Mathcal G}
     \begin{split}
         (\mathcal{G}(\mathbf{v}))_j :=  - \int_{B^c_{\mathbf{x}_0}} [   (\bm{\mu}^{-1} - \mu_0^{-1})\rot A^{(j)} \cdot \rot \ov - \int_{B^c_{\mathbf{x}_0}} i \omega \p \cdot A^{(j)} [(\bm{\sigma} - \sigma_0)  \ov] \\
             +  \int_{\Gamma} [\n \times \mu_0^{-1} \rot A^{(j)}] \cdot \ov.
         \end{split}
     \end{equation}
     \par
     Exploiting this notation, $\eqref{weak problem for Eta}, \eqref{weak formulation for Qstar}$ now respectively read:
     \begin{equation} \label{weak problem for Era 2}
         b[\eta, \xi] = \widetilde{\mathcal{G}}(\xi) \cdot \mathbf{p} \qquad \forall \xi \in W,
     \end{equation}
     and
      \begin{equation} \label{weak formulation for Qstar 2}
         a^{+}[\q, \mathbf{v}] = \mathcal{G}(\mathbf{v}) \cdot \p \qquad \forall \mathbf{v} \in \mathbf{V}.
     \end{equation}
     \par
    \vspace{2mm}
    As a consequence of the fact that the squared norm $|\cdot|^2 : \C \rightarrow \R$ is \textit{nowhere} complex differentiable\footnote{Indeed if $z_0 \neq 0$, \begin{equation*}
        \lim_{z \to z_0} \frac{|z|^2 - |z_0|^2}{z-z_0} = \lim_{z\to z_0}\frac{|z|+|z_0|}{z-z_0}\bigl(|z|-|z_0|\bigr),
    \end{equation*} and the latter limit vanishes if we move along the circle $\{ z : |z| = |z_0| \}$ and is equal to $2 \overline{z}_0$ if we move on the ray $\{ r z_0 : r > 0 \}$.}- exception made for the origin -, we observe that the reduced cost functional $F$ in $\eqref{reduced cost funcional}$ is not Fréchet differentiable. Nevertheless, it admits directional (Gateaux) derivatives at each point $\widehat{\mathbf{p}} \in \R^3$:
    \begin{equation*}
    \begin{split}
        \frac{F(\widehat{\mathbf{p}} + t \mathbf{p}) - F(\widep)}{t} = \nu_E t \int_{B^c_{\mathbf{x}_0}} |\e_{\mathbf{p}}|^2 + \nu_E \operatorname{Re}  \int_{B^c_{\mathbf{x}_0}} (\e_{\widep} - \e_d) \cdot \overline{\e}_{\mathbf{p}} + t \nu_H \int_{B^c_{\mathbf{x}_0}} |\bm{\mu}^{-1} \rot \e_{\mathbf{p}}|^2 \\
        +  \nu_H \operatorname{Re} \left\{ \int_{B^c_{\mathbf{x}_0}} (\bm{\mu}^{-1} \rot \e_{\widep} - \h_d) \cdot \bm{\mu}^{-1} \rot \overline{\e}_{\mathbf{p}}  \right\} + t \nu |\mathbf{p}|^2 +  \nu   \widehat{\mathbf{p}} \cdot \mathbf{p}, 
        \end{split}
    \end{equation*}
    therefore
    \begin{equation*}
    \begin{split}
        \lim_{t \to 0^{+}} \frac{F(\widehat{\mathbf{p}} + t \mathbf{p}) - F(\widep)}{t}& \\ = \nu_E \operatorname{Re} \int_{B^c_{\mathbf{x}_0}} (\e_{\widep} - &\e_d) \cdot \overline{\e}_{\mathbf{p}}  + \nu_H \operatorname{Re} \left\{ \int_{B^c_{\mathbf{x}_0}} (\bm{\mu}^{-1} \rot \e_{\widep} - \h_d) \cdot \bm{\mu}^{-1} \rot \overline{\e}_{\mathbf{p}}  \right\} + \nu   \widehat{\mathbf{p}} \cdot \mathbf{p}
        \end{split}
    \end{equation*}
     for each chosen direction $\mathbf{p}$.
    \par
    \vspace{2mm}
    Hence it follows that the directional derivative
    of the cost functional $F$ in the direction $\mathbf{p} \in \R^3$ at an arbitrary fixed control $\widehat{\p}$ with associated state $\e=\e_{\widehat{\p}}$ is given by:
    \begin{equation} \label{derivative of cost functional}
    \begin{split}
        F'&(\widehat{\mathbf{p}}) \mathbf{p} \\ &=   \operatorname{Re}  \left\{ \nu_E \int_{B^c_{\mathbf{x}_0}} (\e_{\widehat{\mathbf{p}}} - \e_{d}) \cdot \overline{\e}_{\p}  + \nu_H  \int_{B^c_{\mathbf{x}_0}} (\bm{\mu}^{-1} \rot \e_{\widehat{\mathbf{p}}} - \h_{d}) \cdot \bm{\mu}^{-1} \rot \overline{\e}_{\p} \right\}  + \nu  \widehat{\p} \cdot \p.
        \end{split}
    \end{equation}
    Looking at the above expression, we see that the \textit{free} control $\mathbf{p}$ (i.e., the direction) appears implicitly via the mappings $\mathbf{p} \mapsto \overline{\e}_{\mathbf{p}}$ and 
    $\mathbf{p} \mapsto \bm{\mu}^{-1} \rot \overline{\e}_{\mathbf{p}}$, a situation which is usually to be avoided mainly because of how inefficient would be a numerical scheme that requires a PDE solver to act at every iteration. The introduction of an adjoint state is a standard method in optimal control theory to make such dependencies explicit; here the procedure is less straightforward, since we have to somehow take into account the \textit{split structure} of the state variable $\eqref{split structure of the state}$. \par
    To this end, we shall define \textit{two} adjoint states: a vector one and a scalar one, which respectively correspond to $\q$ and $\eta$ in $\eqref{split structure of the state}$.
    \begin{defi}[Adjoint state(s)]
    Let $\widehat{\mathbf{p}} \in \R^3$ be a given control with associated state $\e= \e_{\widehat{\mathbf{p}}}$. The problem to find $({\mathbf{T}},\Psi) \in \mathbf{V} \times W$ such that:
    \begin{equation} \label{weak adjoints}
    \left\{
        \begin{aligned}
        &a^{-}[\mathbf{T}, \mathbf{v}] = \nu_E \int_{B^{c}_{\mathbf{x}_0}} (\e_{\widehat{\mathbf{p}}} - \e_d) \cdot \overline{\mathbf{v}} + \nu_H \int_{B^{c}_{\mathbf{x}_0}} (\bm{\mu}^{-1} \rot \e_{\widehat{\mathbf{p}}} - \h_d) \cdot \bm{\mu}^{-1} \rot \overline{\mathbf{v}}  \quad \forall \mathbf{v} \in \mathbf{V} \\
        &b[\Psi, \xi] = \nu_E \int_{B^{c}_{\mathbf{x}_0}} (\e_{\widehat{\mathbf{p}}} - \e_d) \cdot \nabla \overline{\xi}   \quad \forall \xi \in W.
        \end{aligned}
        \right.
    \end{equation}
    is called \textit{adjoint equation} of the control problem to minimize $\eqref{completecostfunctional}$ subject to $\eqref{maxwellstate1}-\eqref{maxwellstate4}$. The functional spaces $\mathbf{V}, W$ have already been defined respectively in $\eqref{definition of space V}, \eqref{definition of W}$ and $a^{-}[\cdot, \cdot], b[\cdot, \cdot]$ are the sesquilinear forms appearing in the weak formulations for $\q, \eta$: see $\eqref{sesquilinear forms for Qstar}, \eqref{weak formulation for Qstar}$ and $\eqref{weak problem for Eta}$.  
    \end{defi}
    \begin{cor}[Existence of adjoint states]
    For all given target fields $\e_d, \h_d \in \bm{L}^2(\Omega)$, for every fixed control $\widehat{\mathbf{p}} \in \mathcal{P}_{ad}$, the adjoint system $\eqref{weak adjoints}$ has a unique solution $({\mathbf{T}}_{\widehat{\mathbf{p}}}, {\Psi}_{\widehat{\mathbf{p}}}) =: (\widehat{\mathbf{T}},\widehat{\Psi}) \in \mathbf{V} \times W$; $\widehat{\mathbf{T}}, \widehat{\Psi}$ are respectively called first and second \textit{adjoint state} associated with $\widehat{\mathbf{p}}$.
    \end{cor}
    \par
    This result again follows from the Lax and Milgram lemma because the sesquilinear forms on the LHS are coercive in the corresponding spaces. \par
    \vspace{2mm}
   We fix $\widehat{\mathbf{p}} \in \mathcal{P}_{ad}$; testing the weak formulations $\eqref{weak adjoints}$ with respectively $\q_{\mathbf{p} - \widehat{\mathbf{p}}} \in \mathbf{V} \xhookrightarrow{} \bm{H}(\rot ; \Omega)$ and $\eta_{\mathbf{p} - \widehat{\mathbf{p}}} \in W \xhookrightarrow{} H^1(\Omega)$ and summing up the two terms, we get
    \begin{equation} \label{first big chain of equalities}
        \begin{split}
        a^{-}[\widehat{\mathbf{T}}, \q_{\mathbf{p} - \widehat{\mathbf{p}}}] + &b[\widehat{\Psi} , \eta_{\mathbf{p} - \widehat{\mathbf{p}}} ] \\
        &= \nu_E \int_{B^{c}_{\mathbf{x}_0}} (\e_{\widehat{\mathbf{p}}} - \e_d) \cdot \overline{\q}_{\mathbf{p} - \widehat{\mathbf{p}}} +  \nu_E \int_{B^{c}_{\mathbf{x}_0}} (\e_{\widehat{\mathbf{p}}} - \e_d) \cdot \nabla \overline{\eta}_{\mathbf{p} - \widehat{\mathbf{p}}} \\ &\phantom{1234} + \nu_H \int_{B^c_{\mathbf{x}_0}} (\bm{\mu}^{-1} \rot \e_{\widehat{\mathbf{p}}} - \h_d ) \cdot \bm{\mu}^{-1} \rot \overline{\q}_{\mathbf{p}- \widehat{\mathbf{p}}} \\ &= \nu_E \int_{B^{c}_{\mathbf{x}_0}} (\e_{\widehat{\mathbf{p}}} - \e_d) \cdot [\overline{\q}_{\mathbf{p} - \widehat{\mathbf{p}}} + \nabla \overline{\eta}_{\mathbf{p} - \widehat{\mathbf{p}}} ] \\
         &\phantom{1234} + \nu_H \int_{B^c_{\mathbf{x}_0}} (\bm{\mu}^{-1} \rot \e_{\widehat{\mathbf{p}}} - \h_d) \cdot \bm{\mu}^{-1} ( \rot \overline{\e}_{\mathbf{p}- \widehat{\mathbf{p}}} - \rot \overline{A}(\mathbf{p} - \widehat{\mathbf{p}})) \\ 
         &= \nu_E \int_{B^{c}_{\mathbf{x}_0}} (\e_{\widehat{\mathbf{p}}} - \e_d) \cdot [\overline{\e}_{\mathbf{p} - \widehat{\mathbf{p}}} - \overline{A}(\mathbf{p} - \widehat{\mathbf{p}}) ] \\ &\phantom{1234}+ \nu_H \int_{B^c_{\mathbf{x}_0}} (\bm{\mu}^{-1} \rot \e_{\widehat{\mathbf{p}}} - \h_d) \cdot \bm{\mu}^{-1} ( \rot \overline{\e}_{\mathbf{p}- \widehat{\mathbf{p}}} - \rot \overline{A}(\mathbf{p} - \widehat{\mathbf{p}})).
        \end{split} 
    \end{equation}
    On the other hand, the sesquilinear forms $a^{+}[\cdot, \cdot], a^{-}[\cdot, \cdot]$ satisfy
    \begin{equation*}
        \overline{a^{+}[\mathbf{u}, \mathbf{v}]} ={a^{-}[\mathbf{v}, \mathbf{u}]} \qquad \forall \mathbf{u}, \mathbf{v} \in \mathbf{V} 
    \end{equation*}
    while $b[\cdot, \cdot]$ is Hermitian and therefore rearranging the terms in $\eqref{first big chain of equalities}$ it follows that:
    \begin{equation} \label{identity for optimality conditions}
    \begin{split}
       \nu_E \int_{B^{c}_{\mathbf{x}_0}} (\e_{\widehat{\mathbf{p}}} - \e_d) \cdot  \overline{\e}_{\mathbf{p} - \widehat{\mathbf{p}}} + \nu_H \int_{B^c_{\mathbf{x}_0}} (\bm{\mu}^{-1} \rot \e_{\widehat{\mathbf{p}}} - \h_d) \cdot \bm{\mu}^{-1} \rot \overline{\e}_{\mathbf{p}- \widehat{\mathbf{p}}} \\
       = [\overline{\mathcal{G}(\widehat{\mathbf{T}}) + \widetilde{\mathcal{G}}(\widehat{\Psi})}] \cdot (\mathbf{p} - \widehat{\mathbf{p}}) +  \nu_E \int_{B^{c}_{\mathbf{x}_0}} [ (\e_{\widehat{\mathbf{p}}} - \e_d) \cdot \overline{A}(\mathbf{p} - \widehat{\mathbf{p}})] \\ + \nu_H \int_{B^c_{\mathbf{x}_0}} (\bm{\mu}^{-1} \rot \e_{\widehat{\mathbf{p}}} - \h_d) \cdot \bm{\mu}^{-1} \rot \overline{A}(\mathbf{p} - \widehat{\mathbf{p}}),
        \end{split}
    \end{equation}
    where the definitions of $\mathcal{G}, \widetilde{\mathcal{G}}$ are respectively given in $\eqref{Mathcal G}, \eqref{widetilde Mathcal G}$: they correspond to the linear mappings appearing on the right hand sides in the weak formulations for $\eta, \q$, see $\eqref{weak problem for Eta}, \eqref{weak formulation for Qstar}$ and $\eqref{Mathcal G}, \eqref{widetilde Mathcal G}$. \par 
    \vspace{2mm}
    The above expression is not yet completely satisfying since the \textit{free} control $\mathbf{p}$ still somehow appears implicitly in the right hand side of $\eqref{identity for optimality conditions}$. Nevertheless, we can still make use of the adjoint states to overcome this problem. Indeed we have:
    \begin{equation} \label{first term}
        \begin{split}
            \int_{B^{c}_{\mathbf{x}_0}} [ (\e_{\widehat{\mathbf{p}}} - \e_d) \cdot \overline{A}(\mathbf{p} - \widehat{\mathbf{p}})] = \int_{B^{c}_{\mathbf{x}_0}} \overline{A}^{T}(\e_{\widehat{\mathbf{p}}} - \e_d) \cdot (\mathbf{p} - \widehat{\mathbf{p}}) = \left(  \int_{B^{c}_{\mathbf{x}_0}} \overline{A}^{T}(\e_{\widehat{\mathbf{p}}} - \e_d)   \right) \cdot (\mathbf{p} - \widehat{\mathbf{p}}) \\
            = \sum_{i=1}^3 (\mathbf{p}  - \widehat{\mathbf{p}})_{i}    \int_{B^{c}_{\mathbf{x}_0}} \left( \sum_{j=1}^n (\e_{\widehat{\mathbf{p}}} - \e_d)_j (\overline{A}^{T})_{ij} \right) = \sum_{i=1}^3 (\mathbf{p}  - \widehat{\mathbf{p}})_{i} \int_{B^{c}_{\mathbf{x}_0}} (\e_{\widehat{\mathbf{p}}} - \e_d) \cdot \overline{A}^{(i)},
           \end{split}
    \end{equation}
    where $\overline{A}^{(i)}$ denotes the \textit{i}-th column of the matrix $\overline{A}$. Similarly, for the last term in $\eqref{identity for optimality conditions}$ we can write\footnote{In the first equality in $\eqref{second term}$, we use the fact that:
    \begin{equation*}
        \rot(A\mathbf{q}) = \sum_{k=1}^3 q_k \rot A^{(k)} ,
    \end{equation*}
    where $\mathbf{q}$ is a fixed vector of $\R^3$ and $A^{(k)}$ denotes the $k$-\textit{th} column of the matrix $A = A(\mathbf{x})$. Using the Levi-Civita symbol, the LHS can be rewritten as:
    \begin{equation*}
        \rot(A \mathbf{q}) = \partial_i (A_{jl} q_l) \epsilon_{ijk} \mathbf{e}_k = [q_l \partial_i A_{jl} + \partial_i q_l A_{jl} ]\epsilon_{ijk} \mathbf{e}_k = q_l \partial_i A_{jl} \epsilon_{ijk} \mathbf{e}_k ;
    \end{equation*}
    the RHS is equal to 
    \begin{equation*}
        \sum_{l=1}^3 q_l \rot A^{(l)} = q_l \rot A^{(l)} = q_l \partial_i  A^{(l)}_j  \epsilon_{ijk} \mathbf{e}_k,
    \end{equation*}
    on the other hand, $A^{(l)}_j$ is the $j$-\textit{th} component of the column vector $A^{(l)}$, namely $A_{jl}$.
    }:  \begin{equation} \label{second term}
        \begin{split}
            \int_{B^c_{\mathbf{x}_0}} (\bm{\mu}^{-1} \rot \e_{\widehat{\mathbf{p}}} - \h_d) & \cdot \bm{\mu}^{-1} \rot \overline{A}(\mathbf{p} - \widehat{\mathbf{p}}) = \\ &=  \int_{B^c_{\mathbf{x}_0}} (\bm{\mu}^{-1} \rot \e_{\widehat{\mathbf{p}}} - \h_d) \cdot \bm{\mu}^{-1} \sum_{i=1}^3 (\mathbf{p}- \widehat{\mathbf{p}})_i \rot \overline{A}^{(i)} \\
            &= \sum_{i=1}^3 (\mathbf{p} - \widehat{\mathbf{p}} )_i \int_{B^c_{\mathbf{x}_0}} (\bm{\mu}^{-1} \rot \e_{\widehat{\mathbf{p}}} - \h_d) \cdot \bm{\mu}^{-1} \rot \overline{A}^{(i)}.
        \end{split}
    \end{equation}
    The above identities can be now exploited to eventually derive necessary (and sufficient) optimality conditions. Before doing that, let us define by $\mathcal{A}^{(i)}$ a suitable extension in $B_{\mathbf{x}_0}$ of the vector function $A^{(i)}$ whose components $A^{(i)}_j$ are given by:
    \begin{equation*}
        A^{(i)}_j = - i \omega \mu_0 [\Phi_{\mathbf{x}_0} \delta_{ij} + D_i D_j \Phi_{\mathbf{x}_0}].
    \end{equation*}
    Here, for \textit{suitable extension} we mean that $\mathcal{A}^{(i)} \in \bm{H}(\rot ; \Omega)$. Moreover, for each $j=1, \dots 3$, let $u_j \in H^1(\Omega_I)$ be the solution of the following problem:
    \begin{equation*}
        \left\{
        \begin{aligned}
        &\di(\bm{\epsilon}_I \nabla u_j ) = \di(\bm{\epsilon}_I A^{(j)}) \qquad \textnormal{in } \Omega_I \\
        &\bm{\epsilon}_I \nabla u_j \cdot \n = \bm{\epsilon}_I A^{(j)} \cdot \n \qquad \textnormal{on } \Gamma \\
        &u_j = 0 \qquad \textnormal{on } \Gamma_C,
        \end{aligned}
        \right.
    \end{equation*}
    and set
    \begin{equation*}
    \widetilde{u}_j :=
        \left\{
        \begin{aligned}
        &u_j \qquad \textnormal{in } \Omega_I \\
        &0 \qquad \textnormal{in } \Omega_C.
        \end{aligned}
        \right.
    \end{equation*}
    Then by construction
    \begin{equation*}
        \mathcal{A}^{(j)} - \nabla \widetilde{u}_j \in \mathbf{V}
    \end{equation*}
    for each $j=1, \dots 3$, so that $\mathcal{A}^{(j)} - \nabla \widetilde{u}_j$
    is now an admissible test function for $\eqref{weak adjoints}_1$.
\begin{teo}[First order optimality conditions]
Let $\mathbf{p}^{*} \in \mathcal{P}_{ad} \subset \mathbb{R}^3$ be an optimal control for problem $\eqref{completecostfunctional}$ and let $\e_{\mathbf{p}^{*}}$ be the corresponding optimal electric field; then there exists a unique adjoint state $(\mathbf{T}^{*}, \Psi^{*}) \in (\mathbf{V} \times W)$ which solves $\eqref{weak adjoints}$, such that the following inequality holds:
\begin{equation} \label{first order condition}
   \operatorname{Re} \left\{ \overline{\mathcal{G}(\mathbf{T}^{*}) + \widetilde{\mathcal{G}}(\Psi^{*})} + \mathbf{a}^{-}[\mathbf{T}^{*}, \mathcal{A}]+ \mathbf{b}[\Psi^{*}, \widetilde{u} ] + \nu \mathbf{p}^{*} \right\} \cdot (\mathbf{p} - \mathbf{p}^{*}) \ge 0 \qquad \forall \mathbf{p} \in \mathcal{P}_{ad},
\end{equation}
where $\mathcal{G}, \widetilde{\mathcal{G}}$ are defined in $\eqref{Mathcal G}, \eqref{widetilde Mathcal G}$, 
\begin{align*}
    \mathbf{a}^{-}[\mathbf{T}^{*},\mathcal{A}] &:= \begin{bmatrix}
           a^{-}[\mathbf{T}^{*},\mathcal{A}^{(1)} - \nabla \widetilde{u}_1] \\
           a^{-}[\mathbf{T}^{*},\mathcal{A}^{(2)} - \nabla \widetilde{u}_2] \\
            a^{-}[\mathbf{T}^{*},\mathcal{A}^{(3)} - \nabla \widetilde{u}_3 ] \\
         \end{bmatrix}
  \end{align*}
  and
  \begin{align*}
    \mathbf{b}[\Psi^{*},\widetilde{u}] &:= \begin{bmatrix}
           b[\Psi^{*},\widetilde{u}_1]  \\
           b[\Psi^{*},\widetilde{u}_2] \\
            b[\Psi^{*},\widetilde{u}_3] \\
         \end{bmatrix}.
  \end{align*}
Conversely, if inequality $\eqref{first order condition}$ holds for some $\mathbf{p}^{*}$ and $\nu >0$, then $\mathbf{p}^{*}$ is optimal for $\eqref{completecostfunctional}$.
\end{teo}
\begin{proof}
It is well known that for an optimal control $\mathbf{p}^{*}$, the inequality
\begin{equation} \label{General variational inequality}
    F'(\mathbf{p}^{*})(\mathbf{p} - \mathbf{p}^{*}) \ge 0 \qquad \forall \mathbf{p} \in \mathcal{P}_{ad}
    \end{equation}
    holds. The fact that if $\nu > 0$ this variational inequality is both necessary and sufficient follows from the strict convexity of the objective functional. We shall show that $\eqref{General variational inequality}$ is actually equivalent to $\eqref{first order condition}$. The derivative of the cost functional $\eqref{derivative of cost functional}$ evaluated at $\widehat{\mathbf{p}} := \mathbf{p}^{*}$ in the direction $\mathbf{p} := \mathbf{p} - \mathbf{p}^{*}$ reads:
    \begin{equation} \label{derivative of cost functional at optimal point}
        \begin{split}
        F'(&\mathbf{p}^{*})(\mathbf{p} - \mathbf{p}^{*}) \\ &=  \operatorname{Re}  \left\{ \nu_E \int_{B^c_{\mathbf{x}_0}} (\e_{\mathbf{p}^{*}} - \e_{d}) \cdot \overline{\e}_{\p - \p^{*}}  + \nu_H  \int_{B^c_{\mathbf{x}_0}} (\bm{\mu}^{-1} \rot \e_{\mathbf{p}^{*}} - \h_{d}) \cdot \bm{\mu}^{-1} \rot \overline{\e}_{\p - \p^{*}} \right\} \\ \phantom{123456789}&  + \nu  \p^{*} \cdot (\p - \p^{*}).
        \end{split}
    \end{equation}
    Owing to $\eqref{identity for optimality conditions}, \eqref{first term}$ and $\eqref{second term}$, we see that the first two addenda in $\eqref{derivative of cost functional at optimal point}$ are equal to (disregarding the real part operator in front):
    \begin{equation} \label{quasi final step}
    \begin{split}
       &[\overline{\mathcal{G}(\mathbf{T}^{*}) + \widetilde{\mathcal{G}}({\Psi}^{*})}] \cdot (\mathbf{p} - \mathbf{p}^{*}) \\ &+
        \sum_{i=1}^3  (\mathbf{p} - \mathbf{p}^{*})_i \left\{ \nu_E \int_{B^{c}_{\mathbf{x}_0}} (\e_{\mathbf{p}^{*}} - \e_d) \cdot \overline{A}^{(i)} + \nu_H  \int_{B^c_{\mathbf{x}_0}} (\bm{\mu}^{-1} \rot \e_{\mathbf{p}^{*}} - \h_d) \cdot \bm{\mu}^{-1} \rot \overline{A}^{(i)}  \right\}.
        \end{split}
    \end{equation}
    On the other hand, for each $i \in \{1,2,3\}$ we have by $\eqref{weak adjoints}_1$:
    \begin{equation*}
    \begin{split}
    a^{-}[&\mathbf{T}^{*}, \mathcal{A}^{(i)} - \nabla \widetilde{u}_i ] \\
    &=\nu_E \int_{B^c_{\mathbf{x}_0}} (\e_{\mathbf{p}^{*}} - \e_{d}) \cdot (\overline{\mathcal{A}^{(i)} - \nabla \widetilde{u}_i }) + \nu_H  \int_{B^c_{\mathbf{x}_0}} (\bm{\mu}^{-1} \rot \e_{\mathbf{p}^{*}} - \h_d) \cdot \bm{\mu}^{-1} \rot \overline{\mathcal{A}}^{(i)} \\
    &= \nu_E \int_{B^c_{\mathbf{x}_0}} (\e_{\mathbf{p}^{*}} - \e_{d}) \cdot \overline{A}^{(i)} + \nu_H  \int_{B^c_{\mathbf{x}_0}} (\bm{\mu}^{-1} \rot \e_{\mathbf{p}^{*}} - \h_d) \cdot \bm{\mu}^{-1} \rot \overline{A}^{(i)} \\
    &\phantom{12354}- \underbrace{\nu_E \int_{B^c_{\mathbf{x}_0}} (\e_{\mathbf{p}^{*}} - \e_d) \cdot \nabla \overline{\widetilde{u}}_i}_{= b[\Psi^{*}, \widetilde{u}_i]}
    \end{split}
    \end{equation*}
    since $A^{(i)} = \mathcal{A}^{(i)}|_{B^c_{\mathbf{x}_0}}$ by construction. The latter computation together with $\eqref{quasi final step}$ gives the result.
    \end{proof}
    \begin{rmk}
    If $\p^{*}$ lies in the interior of $\mathcal{P}_{ad}$, then by standard argument it can be shown that the explicit formula
    \begin{equation*}
        \mathbf{p}^{*} = - \frac{1}{\nu} \operatorname{Re} \left\{ \overline{\mathcal{G}(\mathbf{T}^{*}) + \widetilde{\mathcal{G}}(\Psi^{*})} + \mathbf{a}^{-}[\mathbf{T}^{*}, \mathcal{A}]+ \mathbf{b}[\Psi^{*}, \widetilde{u} ]  \right\}
    \end{equation*}
    holds.
    \end{rmk}

\par
\vspace{1mm}
\flushleft
\textbf{Acknowledgements.} I am grateful to the PhD school in Mathematics of the University of Trento for its support and funding, and I wish to thank my advisor Alberto Valli for suggesting me this problem as well as for the counteless comments and corrections.

\printbibliography

\end{document}